\documentclass[12pt,oneside, reqno]{amsart}
\usepackage{amssymb}
\usepackage{url,amstext,amsfonts,amssymb,amscd,amsbsy,amsthm,verbatim}
\usepackage{mathrsfs}
\usepackage{ifthen, fullpage}
\usepackage{comment}
\usepackage{color}
\usepackage{amsthm}
\usepackage{latexsym}
\usepackage[all]{xy}
\usepackage{enumerate}
\usepackage{hyperref}
\usepackage[initials]{amsrefs}

\usepackage{graphicx}
\usepackage[justification=centering]{caption} 

\usepackage{thmtools, thm-restate}
\theoremstyle{plain}

\newtheorem{theorem}{Theorem}[section]
\newtheorem{defn}[theorem]{Definition}
\newtheorem{conj}[theorem]{Conjecture}
\newtheorem{prop}[theorem]{Proposition}
\newtheorem{lemma}[theorem]{Lemma}

\theoremstyle{remark}
\newtheorem{remark}[theorem]{Remark}
\newtheorem{example}[theorem]{Example}

\def\phi{\varphi}
\def\rig#1{\smash{ \mathop{\longrightarrow}
 \limits^{#1}}}

\def\F{{\mathcal F}}

\def\g{{\mathfrak g}}
\def\h{{\mathfrak h}}

\newcommand{\defi}[1]{\textsf{#1}} 
\newcommand{\isom}{\cong}

\newcommand{\End}{\operatorname{End}}
\newcommand{\Gr}{\operatorname{Gr}}
\newcommand{\SL}{\operatorname{SL}}
\newcommand{\GL}{\operatorname{GL}}

\newcommand{\sgn}{\operatorname{sgn}}
\newcommand{\Stab}{\operatorname{Stab}}

\newcommand{\rank}{\operatorname{rank}}

\newcommand{\PP}{\mathbb{P}}
\newcommand{\CC}{\mathbb{C}}

\newcommand{\FF}{\mathbb{F}}

\newcommand{\edim}{\operatorname{exp.dim}}
\newcommand{\vdim}{\operatorname{v.dim}}

\def\bw#1{{\textstyle\bigwedge^{\hspace{-.2em}#1}}}

\newcommand{\ccirc}[1]{\xymatrix@1{*+<1ex>[o][F-]{#1}}}

\newcommand{\dalton }[1]{{\color{red} \sf  Dalton: [#1]}}
\newcommand{\luke}[1]{{\color{red} [\sf Luke: [#1]]}}

\author{Dalton Bidleman}
\author{Luke Oeding}\thanks{{\{deb0036,oeding\}}@auburn.edu}
\address{Department of Mathematics and Statistics,
Auburn University,
Auburn, AL, USA
}

\title{Restricted Secant Varieties of Grassmannians}
\date{\today}

\begin{document}

\maketitle

\begin{abstract}
Restricted secant varieties of Grassmannians are constructed from sums of points corresponding to $k$-planes with the restriction that their intersection has a prescribed dimension.  
We study dimensions of restricted secant of Grassmannians and relate them to the analogous question for secants of Grassmannians via an incidence variety construction. We define a notion of expected dimension and give a formula for the dimension of all restricted secant varieties of Grassmannians that holds if the BDdG conjecture \cite{BaurDraismadeGraaf}*{Conjecture 4.1} on non-defectivity of Grassmannians is true. We also demonstrate example calculations in Macaulay2, and point out ways to make these calculations more efficient. We also show a potential application to coding theory.
\end {abstract}

\section{Introduction}
Secant varieties are fundamental objects in algebraic geometry. 
Given a projective variety $X\subset \PP^n$, the $k$-secant variety is the closure of all points that have $X$-rank $\leq k$, i.e. those of the form $[v] = [x_1 + \cdots + x_k]$ with $[x_i] \in X$ for all $i$. Such decompositions have many applications since one can view an $X$-rank decomposition as recovering the information stored in the $[x_i]$ from $[v]$ \cite{KoldaBader}. Often it is not possible to choose the $x_i$ completely independently, and in this article we study one such class of examples called restricted secant varieties, expanding on an idea from Fulton and Harris \cite{FultonHarris}*{Ex.~15.44}, see (see Definition~\ref{def:rrest}).

For $X \subset \PP V$ invariant under the action of a subgroup $G \subset \GL(V)$  secant varieties of $X$ inherit this $G$-invariance.
Hence secant varieties can be part of a classification of orbits \cites{Vinberg-Elasvili,Antonyan, Antonyan-translation,AOP_Grassmann}. One seeks easy ways to compute invariants that permit the separation of orbits, perhaps the first of which is dimension.

Terracini's lemma \cite{Zak} reduces the dimension of the secant variety of a variety $X$ to a dimension count for a sum of linear spaces. This  count is usually correct (as long as the spaces don't intersect), so when it fails for the $k$-secant variety one says that $X$ is $k$-defective. Alexander and Hirshowitz \cite{AOP_Grassmann} settled the classification of defectivity for Veronese re-embeddings of projective space (for recent proofs see  \cites{Postinghel, OttavianiRubei}). Initial studies in the cases of general tensors can be found \cite{Geramita_Lectures}, which led to conjectures \cites{AOP_Grassmann, BaurDraismadeGraaf} and further progress has been made in verifying initial cases in \cites{abo2018most,CHIO}. 
The story is similar in the case of secants to Grassmannians, see  \cite{CGG6_Grassmann} and  \cite{BaurDraismadeGraaf}. 

We focus on restricted secants to Grassmannians. Our main result is the following: 

\begin{restatable}{theorem}{fiber}
\label{thm:fiber}
Let $\dim( V)=n$ and $r,s,\geq 0$ and $0 \leq k \leq n$. Then the restricted secant variety
$\sigma_s^r(\Gr(k,V))$ is birationally isomorphic to the fiber bundle, denoted $\Xi$, with base $\Gr(r,V)$ and whose fiber over a point $E$ is   $\sigma_s(\Gr(k-r, V/E))$.
\end{restatable}

This gives a way to calculate the dimensions of  restricted secants of Grassmannians and determine the defective cases as well. Interest in the skew-symmetric case is partly driven by the connections to coding theory which are explored in   \cite{AOP_Grassmann}*{Sec.~3}. Continuing in this vein, in Section~\ref{sec:coding} we provide a possible application to coding theory for restricted secants.

In Section~\ref{sec:notation} we recall basic definitions and preliminary notions. In Section~\ref{sec:dim} we recall standard techniques for computing dimensions of parametrized varieties, we sketch our implementation in Macaulay2,  \cite{M2} and discuss improvements to computational efficiency for secants and restricted secants. In Sections~\ref{sec:1restricted} and \ref{sec:rrestricted} we respectively consider the 1-restricted and $r$-restricted secant varieties of Grassmannians and compute their dimensions. In Section~\ref{sec:dimr} we provide a fiber bundle construction that we use to study the dimensions of the restricted secant varieties. We also propose the complete description of dimensions for restricted secants of Grassmannians based on the BDdG conjecture \cite{BaurDraismadeGraaf}*{Conjecture 4.1}:

\begin{restatable}{cor}{classification}
\label{thm:classification}
If the BDdG conjecture is true, then $\sigma_s^{r}(\Gr(k,V))$ has no additional defect other than the defect coming from (usual) secant varieties of Grassmannians.
\end{restatable}

\section{Preliminaries and Notation}\label{sec:notation}
\subsection{Grassmannians}
Let $V,W$ denote complex finite-dimensional vector spaces. Given a projective variety $X \subset \PP V$ let $\widehat X$ denote the cone in $V$.
The Grassmann variety is the collection of $k$-dimensional subspaces of $V$ denoted  $\Gr(k,V)$, or $\Gr(k,n)$ if the ambient space $V$ is $n$-dimensional. The Pl\"ucker embedding maps $\Gr(k,V)$ into $\PP \bw{k} V$ as follows. Given a $k$-plane $E$, select a basis $e_{1},\ldots, e_{k}$ of $E$ and send it to the class of the wedge product $[e_1\wedge \cdots \wedge e_k]$. We will write  $\widehat E = e_1\wedge \cdots \wedge e_k $ for a representative on that line. One checks that this map is well-defined independent of the choice of basis of $E$, and that it is an embedding.

Since we work in projective space and have a skew-symmetric product we often insist that $i_{1}<i_{2}<\dots<i_{k}$. The general linear group acts transitively on the set of $k$-planes, hence the Grassmannian is also the $\GL(V)$-orbit
\[
\Gr(k,V) = \GL(V).[e_{1}\wedge \cdots \wedge e_{k}].
\]

Any nonzero element $\delta \in \bw{n}V$ induces an isomorphism $\bw{k}V \to \bw{n-k}V$ given by contraction on simple elements and extending through linearity. For instance, if $\delta = e_1\wedge \cdots \wedge e_n$, then $\delta (e_I) = \sgn(I, I^\star) e_{I^\star} $, where $\star$ denotes complement on multi-indices, and $\sgn(I, I^\star)$ denotes the sign of the corresponding permutation of $[n]$.  This induces a duality on Grassmannians
\begin{equation}\label{eq:adual}
\Gr(k,V) \isom \Gr(\dim(V)-k, V), \quad \text{or} \quad \Gr(k,n) \isom \Gr(n-k,n).
\end{equation}

For parametrized varieties the differential-geometric view of the tangent space is useful:
\begin{defn}\label{def:tgt}
Let  $x\in X$ be a smooth point on an algebraic variety $X \subset \PP W$. The cone over the tangent space to $X$ at $x$ is
\[\widehat{ T_{x}}X =\{\gamma'(0)\mid \gamma\colon \CC^1 \rig{} X , \gamma(0)=x\}.\]
\end{defn}

It is a standard exercise in
 \cite{LandsbergTensorBook}*{Ch.~6} for instance
 to verify the following expression:
 \[
 \widehat{T_E}\Gr(k,V) = E + E^*\otimes V/E,
 \]
 where $E^*$ is the dual vector space, and $V/E$ is the quotient. One finds other useful characterizations of this tangent space in \cite{CGG6_Grassmann}.
We prefer the following description. 
The tangent space to the Grassmannian at $E$ is spanned by $e_{1}\wedge \cdots \wedge e_{k}$ and all square-free monomials of the form
$e_{I\setminus \{i\} \cup \{j\}}$,
where $I = \{1,\dots,k\}$, $i \in I$ and $j \in \{k+1,\dots,n\}$. 
This description has an interpretation using simplices. Recall that the set of multi-indices of length $k$, denoted $\mathbb{S}_k = \{J \subset[n] \mid |J| = k\}$, parametrizes the space of $k$-simplices. There is a discrete distance function called the Hamming distance $d_H$ on $\mathbb{S}_k$, which is defined as $d_H(I,J)$ the size of the symmetric difference of $I$ and $J$. 

The indices that occur in the monomials in $\widehat{T_E}\Gr(k,V)$ correspond to all simplices in a Hamming ball of radius $1$ centered at the standard $k$-simplex, which one can show contains $k(n-k)+1$ simplices.

\subsection{ \texorpdfstring{$X$}{}-Rank, secants and restricted secants}

Given a variety $X \subset \PP V$, 
the $X$-rank of a point $[p] \in \PP V$ is the minimal number $s$ of points $[x_{1}],\dots, [x_{s}]$ such that $p$ lies in the span of the $x_i$. 
One notion of tensor rank (the CP rank) is defined when $X$ is the variety of rank-1 tensors (indecomposable  tensors). The $s$-secant variety of $X$, denoted $\sigma_{s}(X)$, is the Zariski closure of the points of $\PP V$ with $X$-rank $s$. Points in $\sigma_s(X)$ are said to have $X$-border rank $s$. While $X$-rank is not semi-continuous, $X$-border rank is semi-continuous by construction. 

Restricted secant varieties of Grassmannians generalize \cite{FultonHarris}*{Ex.~15.44} as follows:
\begin{defn}\label{def:rrest}
The $r$-restricted $s$-secant variety of $\Gr(k,V)$, is $\sigma_{s}^{r}(\Gr(k,V)) = $
\[
\overline{\left\{
[\lambda_1 \widehat E_1+\cdots + \lambda_s  \widehat E_s] \mid E_i \in \Gr(k,V), [\lambda]\in \PP^{s-1}, \dim({\textstyle \bigcap}_{i=1}^s E_i) \geq r
\right\}} \subset \PP \bw k V
.\]
\end{defn}
Note that it is necessary to define the dimension of the intersection as being $\geq r$ rather than $=r$ to ensure the variety is non-empty.
When more than $2$ $k$-planes are involved the intersection structure is more complicated, and is not, in general characterized by a single number. We find it already interesting to study this case,

\subsection{Inheritance and orbit stability}

Given a family $\mathcal{F}$ of algebraic varieties one can ask what properties a variety inherits from its subvarieties coming from the same family. For example, we define an \defi{orbit family} $\F = \F(W_\bullet, G_\bullet, X_\bullet)$ by the data: a chain of vector subspaces $W_\bullet = W_0\subset \cdots \subset W_i \subset \cdots \subset W_n$, a family of groups $G_\bullet$ with $G_i \subset \GL(W_i)$ and a family of varieties $X_\bullet$ with $X_i \subset \PP W_i$ and we require the property that $G_j.X_i \subset X_j$ whenever $i\leq j$. We say that \defi{orbit stability} occurs at step $p$ if $G_j.X_i = X_j$ whenever $p\leq i\leq j$ \cite{CGG1_tensors}. When orbit stability occurs, we can use this structure to compute the dimensions of the $X_i$ for $i\geq p$ precisely. Specifically, when the varieties $X_i$ are defined by the closure of a single orbit, orbit stability at step $p$ implies that all $X_i$ for $i\geq p$ have the same normal  form (representative of an orbit on a full-dimensional open set) $n\in X_i$. Hence, we can describe the tangent spaces to the $X_i$ at $n$ as follows for all $i\geq p$:
\[
\widehat {T}_n X_i = \widehat {T}_n X_p + \text{a correction term}
\]

To determine this correction term, we recall the following version of an orbit-stabilizer theorem for subspaces (see \cite{LandsbergTensorBook}*{6.9.4} for the case when $V$ is a line, and $G.V = G/P$ is a homogeneous variety). 
\begin{prop} \label{prop:orbstab}
Let $G$ be a connected compact complex semisimple Lie group contained in $\GL(W)$. 
Given a $G$-module $V\subset W$,  set $H = \Stab_G(V)$. Then 
\begin{equation}
\dim( G.V) =\dim(G/H)+\dim(V)
.\end{equation}
\end{prop}

\begin{proof}
The orbit $G.V$ can be seen as a parametrization:
\[\begin{matrix}
G\times V &\to& W \\ 
(g,v) & \mapsto & g.v
\end{matrix}
\]
The tangent space at $v = \text{Id}.v$ can be computed via the Lie algebra action:
\begin{equation}\label{eq:bracketTgt}
\widehat{T_{v}}G.V = v + [\g,V]
,\end{equation}
see \cite{HolweckOeding22}*{Prop.~3.18}.
The decomposition $W = V \oplus V/W$ induces a decomposition of the endomorphisms:
\[
\End(W) = W^*\otimes W = (V^* \otimes V) \oplus (V^* \otimes W/V) \oplus (W/V^* \otimes V) \oplus (W/V^* \otimes W/V) 
.\]
Define the corresponding subspaces $\g_{ij}$ of $\g \subset \End(W)$ via restriction. Seen as a matrix,
\[\g = \begin{pmatrix}
\g_{00} & \g_{01} \\
\g_{10} & \g_{11} 
\end{pmatrix}.
\]
In addition, $\h = \g_{00}\oplus \g_{01} \oplus \g_{11}$ is the subalgebra of $\g$ that stabilizes $V$.  Hence
\begin{multline}
[\g,V] = [\g_{00}\oplus \g_{10} \oplus  \g_{01} \oplus \g_{11}, V] = 
[\g_{00},V]\oplus [\g_{10},V] \oplus [\g_{01},V] \oplus [\g_{11}, V] 
\\
=[\g_{00},V]\oplus [\g_{10},V]  .
\end{multline}
Since $V$ is a $G$-module it is also a $\g_{00}$-module and $[\g_{00},V] = V$. 
 Moreover $v\in V$ so $v+[\mathfrak g, V] = v + [\g_{00},V]\oplus [\g_{10},V]   = V \oplus [\g_{10},V]  $.
Finally, $\g/\h = \g_{10}$, so $\dim(G/H) = \dim ([\g_{10},V])$. 
\end{proof}

We can then apply this to the family of varieties with symmetry.
\begin{prop} \label{prop:biratstab}
Suppose an orbit family $\F$ achieves orbit stability at step $p$, and that $G_i$ acts transitively on the set of $\dim(V_p)$-planes for each $i\geq p$. Define a fiber bundle $\Xi \to \Gr(\dim (V_p), V_i)$ with each fiber over $E \in \Gr(\dim (V_p), V_i)$ equal to a copy of $X_p \subset \PP E$.

Then for all $i\geq p$, $X_i$ is birational to the total space of $\Xi$, and in particular
\begin{equation}
\dim (X_i) = \dim (X_p) + \dim(\Gr(\dim V_p,V_i))
.\end{equation}
\end{prop}
\begin{proof}
We will show that $X_i$ is bi-rational to the fiber bundle $\Xi$ defined in the statement. Then the total space of $\Xi$ has dimension equal to the dimension of the general fiber plus the dimension of the base, or $\dim( X_p) + \dim(\Gr(\dim V_p,V_i))$, so the ``moreover'' part follows.

Let $[x] \in X_i$ be a general point, so we can assume $x$ is on the orbit $G_i.x_i$. Because of orbit stability at step $p$ we can take $x$ to be the normal form for $X_p$, $x =x_p\in X_p$.
Consider the vector space $\widehat{T}_{x_p} X_p$, and take its orbit under the action of $G_i$. Since $G_i$ acts transitively on $\dim (X_p)$ planes, this orbit is $\Gr(\dim(V_p), V_i)$, so we can send $x_p$ to the pair $(\widehat{T}_p X_p, x_p)$, this is a rational mapping. 

For the other direction, suppose we have a pair $(E,x) \in \Xi$ with $x \in \tilde X_p$, where $\tilde X_p$ denotes a copy of $X_p$ in $E$. Then by the assumption that $G_i$ acts transitively on $\dim(V_p)$-planes we can assume that the linear space $E$ is a $G_i$ translate of $\widehat{T}_{x_p} X_p$, hence $x\in g.X_p \subset g.\widehat{T}_n X_p$, with $g\in G_i$.
What is left to show is that the composition of the two maps is the identity. Let $[x] \in X_i$ be a general point. Because of orbit stability, $x=x_p \in X_p$. Apply the first map. Take the orbit of $x_p$ under the action of $G_i$. This produces a pair  $(\widehat{T}_{x_p} X_p, x_p)$. Then, it is true that $x_p \in X_p$ and not just $\tilde X_p$, where $\tilde X_p$ was a copy of $X_p$ in $E$. Further,  since  $G_i$ acts transitively on $\dim (V_p)$-planes, applying the second map and acting on $\widehat{T}_{x_p} X_p$ by $G_i$ means $x_p \in g.\widehat{T}_{x_p} X_p$. But we had $x_p=x$, therefore we arrive back at $[x]$.
\end{proof}

We're interested in the case for fixed $k,r,s$ with  $W_i = \bw{k}V_i$, with $V_0 \subset \cdots \subset V_n $ and $V_i \cong \CC^i$, $G_i = \GL(V_i)$ and $X_i = \sigma_s^r\Gr(k,V_i)$. We denote this family by $\mathcal{G}(r,s,k) =(\bw{k}V_\bullet, \GL(V_\bullet), \sigma_s^r \Gr(k,V_\bullet))$. 
These varieties are defined as orbit closures, and in this particular case orbit stability implies stability of normal forms.

\begin{prop} \label{prop:famstab}
The family $\mathcal{G}(r,s,k)$ obtains orbit stability (at least) when  $p=r+s(k-r)$.
\end{prop}
\begin{proof}
Note the condition that $\mathcal{G}(r,s,k)$ obtains orbit stability at step $p$, where $p=\dim(V_{p})$ and $\sigma_{s}^{r}(\Gr(k,V_{p}))$ can be guaranteed at the first instance where there are enough linearly independent basis vectors to define a general point $x \in \sigma_{s}^{r}(\Gr(k,V_{p}))$ with no additional intersection.
Now count independent parameters. For a given, $r,s,k$ there is an $r$-dimensional overlap which accounts for $r$ elements  $e_i \in V$  and additionally each of the $s$ copies of the Grassmannian requires $k-r$ more $e_i$ elements for a total of $r+s(k-r)$. These $e_i \in V$ can be chosen independently if $n\geq p = r+s(k-r)$.
\end{proof}

\begin{prop} \label{prop:dimfam}
Suppose $\mathcal{G}(r,s,k)$ attains orbit stability at step $p$. For all $n\geq p$ we have
\begin{multline}
\dim( \sigma_s^{r}\Gr(k,V_n) )= \dim (\sigma_{s}^{r}\Gr(k,V_p))+\dim (\Gr(p,n) )
\\
= r(p-r)+s((k-r)(p-k))+s-1+p(n-p).
\end{multline}
\end{prop}
\begin{proof}
When orbit stability occurs there exists a birational morphism:
\[
G_j \times X_j \dashrightarrow X_j
\]
Using the orbit-stabilizer theorem \ref{prop:orbstab} we obtain a dimension count: $r(p-r)+s((k-r)(p-k))+s-1+p(n-p)$.
\end{proof}

\begin{example}
Apply this argument above to the 1-restricted case for  $\sigma_{s}^{1}(\Gr(3,V))$. 
We have the following chain of inclusions. 
\[
\Gr(k,a) \subset \PP \bw{k} \CC^{a} \subset \PP \bw{k} \CC^{n}.
\]
Then, $\sigma_{s}^{1}(\Gr(k,V))$ can be found by taking the appropriate orbits of $\sigma_{s}^{1}(\Gr(k,a))$:
\[\GL(n).\sigma_{s}^{1}(\Gr(k,a)) \subset \sigma_{s}^{1}(\Gr(k,n)).\]
Consider $\sigma_{3}^{1}(\Gr(3,n))$, and take $a=7$ in this case. So, 
$\sigma_{3}^{1}(\Gr(3,V))$ is birational to $\Gr(7,V)) \times \sigma_{3}^{1}(\Gr(3,7))$. Therefore, 
\[
\dim(\sigma_{3}^{1}(\Gr(3,n)))= \dim (\Gr(7,n))+\dim(\sigma_{3}^{1}(\Gr(3,7))).
\]
So, 
$\dim(\sigma_{3}^{1}(\Gr(3,n)))= 7\cdot(n-7)+31=7n-18$ for $n\geq 7$. 

This leads to other explicit formulas for $1$-restricted chordal varieties such as:
\[
\dim(\sigma_2^1(\Gr(4,n))) = 7n-24, \quad \text{for} \quad n \geq 7,
\]
and
\[
\dim(\sigma_2^1(\Gr(5,n))) = 9n-40, \quad \text{for} \quad n \geq 9.
\]
\end{example}

\section{Dimension Calculations}\label{sec:dim}
\subsection{General setup}   
A standard method to compute the dimension of a parametrized projective variety is via differentials.
Recall a \defi{parametrization} is a rational mapping
\[
\phi\colon \PP^M \to \PP ^N
,\]
defined on a non-trivial open subset $U \in \PP^M$. 
That the map $\phi$ is rational  means that $[\phi(x)] = [\phi_0(x):\ldots: \phi_N(x)]$ with $\phi_i(x)$ a rational function for each coordinate $i$.
Recall that the \defi{image} $X$ of a rational mapping $\phi$ is the Zariski closure $\overline{\varphi(U)}$ and note that the definition doesn't depend on which non-trivial open subset we choose as long as $\phi$ is defined on that set. Work on the cone over $U$ and take the total differential (the Jacobian):
\[
d\phi \colon \widehat U \to \CC^{N+1},
\]
noting that $\widehat{T}_{p} U = \CC^{M +1}$, and $\widehat{T}_{\phi(p)}\CC^{N+1} = \CC^{N+1}$.
At a point $[p] \in U$, the linear mapping 
\[d\phi_p \colon
\CC^{M+1}
\to \CC^{N+1}
,\]
may be represented by a matrix with $(i,j)$ entry $\frac{\partial \phi_i}{\partial x_j}(p)$, with $0\leq i\leq N$ and $0\leq j \leq M$. 
The image of $d\phi_p$ for general $[p]\in U $ is the the tangent space $\widehat{T}_{\phi(p)} X$, and hence the rank of $d\phi_p$ computes the dimension of the cone $\widehat X$. 

In summary, to compute the dimension of a parametrized variety we may
\begin{enumerate}
    \item Generate sufficiently many random points $[p]$ of the source.
    \item Compute the partial derivatives $\frac{\partial \phi_i}{\partial x_j}(p)$ and populate the matrix $d\phi_p$.
    \item Compute the rank of the matrix $d\phi_p$.
\end{enumerate}

\subsection{Computing dimensions of  \texorpdfstring{$\sigma_s^r\Gr(k,n)$}{}}
We verified the calculation of dimension for several families of restricted chordal varieties with Macaulay2 \cite{M2}. An example computation can be found in the ancillary files associated with the arXiv version of this article.

Since any $k$-dimensional subspace of an $n$-dimensional space can be represented as the row space of a $k\times n$ matrix, a parametrization for $\sigma_s(\Gr(k,n))$ is given by
\[
\phi\colon \PP (\CC^{k\times n})^{\times s} \to \PP \bw k \CC^{n}
,\]
which takes an $s$-tuple of $k\times n$ matrices (up to scale) to the sum of their vectors of $k$-minors. The open set we work on is the one where all of the matrices in question have full rank.

The Jacobian at a point $p$ is a linear mapping
\[
d\phi\colon  (\CC^{k\times n})^{\times s} \to  \bw k \CC^{n}
,\]
whose coordinates are evaluations of derivatives of sums of minors. Its size is $\binom{n}{k} \times (kns) $. We write $d\phi(A)$ (and similar) to indicate a symbolic Jacobian, and $d\phi_C$ to indicate the evaluation at a point parametrized by an $s$-tuple of matrices $C$.

Similarly, a parametrization $\phi^r$ for $\sigma_s^r(\Gr(k,n))$ is given by restricting the source of $\phi$ to a set where the $s$-tuple of matrices mutually share $r$ row vectors. This restricted source is
\[
(\CC^{r\times n})\times 
(\CC^{(k-r)\times n})^{\times s},
\]
where the first factor is the shared rows. So the Jacobian $d\phi^r$ has size  $\binom{n}{k} \times (rn + (k-r)ns) $. 

Focus on the case $s =2$ for the moment, the case of general $s$ is similar. Given two symbolic matrices $A$ and $B$ in $\CC^{k\times n}$ with the first $r$ rows of $B$ the same as those of  $A$ (to reflect the overlap in their row spaces) we can represent the structure of the sum of the Jacobians of their Pl\"ucker images. 
So $d\phi(A+B) = d\phi(A) + d\phi(B)$. 
Let $A = (a_{ij})$ with $0 \leq i \leq k-1$, $0 \leq j \leq n-1$ and
$B = (b_{ij})$, with $ 0 \leq i \leq k-r-1$, $0 \leq j \leq n-1$. Let $d=\binom{n}{k}$ and $A_I$ represent the maximal minor of $A$ with columns $I$ and order the multi-indices $I$ lexicographically and re-name them $m_1,\ldots, m_d$. The Jacobians of the Pl\"ucker maps of $A$ and $B$ are the following. 
\[
  d\phi(A) = 
    \begin{pmatrix} 
       \dfrac{\partial A_{m_{1}}}{\partial a_{00}} 
       & \cdots &
       \dfrac{\partial A_{m_{d}}}{\partial a_{00} }
       \\
 \vdots & \ddots & \vdots \\
       \dfrac{\partial A_{m_{1}}}{\partial a_{(k-1)(n-1)}}
       & \cdots &
      \dfrac{\partial A_{m_{d}}}{\partial a_{(k-1)(n-1)}} 
      \end{pmatrix}
\]
\[
  d\phi(B) = 
    \begin{pmatrix}
     \dfrac{\partial B_{m_{1}}}{\partial a_{00}} & \cdots &
     \dfrac{\partial B_{m_{1}}}{\partial a_{(r-1)(n-1)}} &
     \dfrac{\partial B_{m_{1}}}{\partial b_{00}} &
     \dots &\dfrac{\partial B_{m_{1}}}{\partial b_{(k-r-1)(n-1)}} 
      \\
       \vdots & \ddots & \vdots &\vdots& \ddots & \vdots \\
     \dfrac{\partial B_{m_{d}}}{\partial a_{00}} & \cdots &
     \dfrac{\partial B_{m_{d}}}{\partial a_{(r-1)(n-1)}} &
     \dfrac{\partial B_{m_{d}}}{\partial b_{00}} &
     \dots &\dfrac{\partial B_{m_{d}}}{\partial b_{(k-r-1)(n-1)}} 
      \end{pmatrix}^\top     
\]
Therefore, $d\phi(A)+d\phi(B) = $
\begin{equation}\label{eq:blockJacobian}\resizebox{0.91\hsize}{!}{%
 $ \begin{pmatrix}
     \dfrac{\partial A_{m_{1}}}{\partial a_{00}} +
     \dfrac{\partial B_{m_{1}}}{\partial a_{00}} & \cdots &
     \dfrac{\partial A_{m_{1}}}{\partial a_{(r-1)(n-1)}} +
     \dfrac{\partial B_{m_{1}}}{\partial a_{(r-1)(n-1)}} &
     \dfrac{\partial A_{m_{1}}}{\partial a_{(r-1)(n-1)+1}} 
       & \cdots &
     \dfrac{\partial A_{m_{1}}}{\partial a_{(k-1)(n-1)}} &
     \dfrac{\partial B_{m_{1}}}{\partial b_{00}} & \cdots &
     \dfrac{\partial B_{m_{1}}}{\partial b_{(k-r-1)(n-1)}} 
     \\
       \vdots &\ddots & \vdots & \vdots & \ddots & \vdots & \vdots& \ddots & \vdots \\
     \dfrac{\partial A_{m_{d}}}{\partial a_{00}} +
     \dfrac{\partial B_{m_{d}}}{\partial a_{00}} & \cdots &
     \dfrac{\partial A_{m_{d}}}{\partial a_{(r-1)(n-1)}} +
     \dfrac{\partial B_{m_{d}}}{\partial a_{(r-1)(n-1)}} &
     \dfrac{\partial A_{m_{d}}}{\partial a_{(r-1)(n-1)+1}} 
       & \cdots &
     \dfrac{\partial A_{m_{d}}}{\partial a_{(k-1)(n-1)}} &
     \dfrac{\partial B_{m_{d}}}{\partial b_{00}} & \cdots &
     \dfrac{\partial B_{m_{d}}}{\partial b_{(k-r-1)(n-1)}} 
      \end{pmatrix}^\top$
      }.
\end{equation}
We use this block structure to make our computations more efficient.

We generate a collection $C = (C_1,\ldots, C_s)$ of random matrices $C_i \in  \CC^{k\times n}$ with the appropriate overlap of their row spaces. Via Terracini's Lemma the Jacobian $d\phi_C$ is the sum  of the differentials of the Pl\"ucker maps, $d\phi_C(A_i)$. The rank of the resulting matrix is equal to the dimension of that restricted chordal variety (as long as the initial choice of $C$ was sufficiently general, which it will be with probability 1). 

\subsection{Computing the dimension of secants in \texttt{M2}}
A naive implementation to compute the dimension of a secant variety of the Grassmannian in \texttt{M2} is given below:
\begin{verbatim}
testnk = (n,k) -> (
  R = QQ[a_(0,0)..a_(k-1,n-1),b_(0,0)..b_(k-1,n-1)];
  A = transpose genericMatrix(R, a_(0,0), n,k);
  B = transpose genericMatrix(R, b_(0,0), n,k);
  fun = matrix{apply(subsets(n,k), s-> det A_s + det B_s )};
  jac = diff(transpose basis(1, R), fun);
  val = map(QQ,R, random(QQ^1,QQ^(dim R)));
  rank val jac
)
\end{verbatim}

The first $3$ lines define the source variables (ring) and matrices. It then defines the mapping, \texttt{fun}, and differentiates with respect to the column vector of variables to calculate the Jacobian as a matrix. Note $\texttt{M2}$ also has a command for Jacobian.

This gives us a set of polynomials which we can evaluate and then find the rank of the corresponding numerical matrix. Note that in this case with negligible computational time we see that the rank is $26$, which is indeed the dimension of the cone over the secant of the Grassmannian $\sigma_2(\Gr(3,7))$.

One can modify the procedure as follows to handle the restricted secant case:
\begin{verbatim}
n = 7; k=3; r=1;
R = QQ[a_(0,0)..a_(k-1,n-1),b_(0,0)..b_(k-r-1,n-1)];
A = transpose genericMatrix(R, a_(0,0), n,k);
B = A^{0..r-1}||transpose genericMatrix(R, b_(0,0), n,k-r);
\end{verbatim}

Note there are fewer variables needed because of the overlap, and we force the matrices to share an $r$ dimensional overlap (the first $r$ rows).
The exact same functions as before compute the Jacobian and its rank.
 For example, in the case of $k=3$,  $n=7$, $r=1$ we find the dimension of the cone over $\sigma_2^1(\Gr(3,7))$ is $20$. We tested this straightforward calculation for $r=1,2$ and $k,n=2,\ldots,10$, as well as for $s=3$, $r=1$, $k,n=2,\ldots,10$.

\subsection{Computational efficiency}

The above naive implementation for calculating the dimension of the restricted chordal variety is not efficient enough to handle larger computations. There is a trade-off of easy-to-implement formulas that ignore redundancy versus more careful implementation that is aware of these redundancies. In addition, we should pay attention to the order of operations for evaluation, in order to limit the size of intermediate computations.

In the naive implementation we take $s$ symbolic matrices with the required $r$-dimensional overlap, and for each of those matrices determine the symbolic Jacobian, and then evaluate at a random point. 
However, the corresponding computation of differentials of minors is very inefficient for even very small cases. For example, a case as small as $s=2$,$r=1$,$k=8$,$n=10$ has Jacobian consisting of more than $100,000$ total terms and takes at least $20$ minutes on a local system to evaluate. This inefficiency can be avoided noting redundancies from the fact that the differential of a minor is a linear combination of smaller minors and representing these entries as \emph{unevaluated} determinants, or subfunctions, (rather than sums of monomials).

This point is illustrated by the following. Suppose $A$ is a matrix of variables, and $\phi$ is the determinant function. Compute the Jacobian of $\phi$ in this case. We don't need to expand a determinant and then take derivatives in order to find an expression for the derivative $\frac{\partial \det A}{\partial a_{ij}}$. Instead use Laplace expansion on the $i$-th row,
\[
\det(A)=\sum_{j=1}^{n}a_{ij}C_{ij} \quad \implies \frac{\partial\det(A)}{\partial a_{ij}}=C_{ij}
\]
where $C_{ij}$ is the cofactor corresponding to the entry, which does not use the variable $a_{ij}$.
Now we can treat $C_{ij}$ as an unevaluated subfunction.
For sums of determinants this same general principle can be applied. Every entry of equation~\ref{eq:blockJacobian} has this format.

Another efficiency consideration is order of operations, particularly evaluation and minor determinants. Generally, it is better to compute the determinant of a numerical matrix instead of evaluating it determinant at a point.

Return to our example. The coordinate functions are (sums of) minors, and hence their partial derivatives are also (sums of) minors. Realizing this allows us to define the Jacobian with subfunctions that evaluate these minors rather than compute the minors as derivatives. Computing a vector of minors at a point allows one to make use of reductions like Gaussian elimination which speed the computation of determinants greatly (on the order of $n^3$ operations rather than $n!$).

To implement this idea, we wrote functions (essentially linear combinations of determinants) to populate the entries of the Jacobian, instead of relying on functions from \texttt{M2} like \texttt{diff} or \texttt{jacobian}.
To replace our use of \texttt{diff} we explicitly populated the Jacobian matrix utilizing appropriate subfunctions (cofactors) depending on the row and column labels. The Jacobian has column labels representing differentiation with respect to variables and row labels representing maximal minors. 

Specifically, we populate this matrix utilizing these rules: in row $m_i$ and column $x_{ij}$ we put a $0$ if maximal minor $A_{m_i}$ does not contain the variable $x_{ij}$, or we put the (numerical) determinant of the $A_{m_i}$ cofactor. This procedure, for each of the $s$ matrices, defines the Jacobian with respect to the collection of variables defined only by that individual matrix and not the variables defined by every one of the $s$ matrices. This directly produces the block structure (seen at \eqref{eq:blockJacobian}) of the Jacobian.
Then we add the numerical Jacobians of each of the matrices together and  calculate the rank. 

Here is an implementation of this strategy as a function that eats a matrix $M$ and spits out the column of the Jacobian of the Pl\"ucker map at $M$ corresponding to the differential with respect to variable $(i,j)$ for $M$.  
\begin{verbatim}
par = (j, s) -> (
c=0;
for i to (length(s) -1) do( 
    if j ==s_i then return i 
    else continue;); 
    return c);
\end{verbatim}
The function \texttt{par}  determines the sign of the cofactor. 
We loop over the subsets representing the maximal minors. The if-then statement determines whether or not the given minor contains the variable at $(i,j)$ and sends it to $0$ if it doesn't otherwise it evaluates the necessary numerical cofactor. 
We compute the full numerical Jacobian by using the function \texttt{dM} to populate the relevant non-zero columns.
\begin{verbatim}
t = set(0..k-1);
dM = (i,j,M)->apply(subsets(n,k),s->
   if not member(j,s) then 0 
   else (-1)^(par(j,s)+i)*
     det(submatrix(M,toList(t-set{i}),toList(s-set{j})))); 
\end{verbatim}

\begin{remark}
With these changes we notice the following differences in speed (on a laptop) for computing the dimension of $\sigma_2^1(\Gr(8,10))$: after $20$ minutes we force the code for the naive implementation to end with no answer, while the new code calculated the dimension in $.09$ seconds.
\end{remark}

\section{Dimensions of 1-Restricted Chordal Varieties}\label{sec:1restricted}
The main tool used to calculate the dimension of the secant variety is:

\begin{lemma}[Terracini\cite{Terracini}] \label{lemma:Terracini}
Suppose $X\subset \PP W$ is an algebraic variety and suppose $[x_{1}],\dots,[x_{s}]$ are smooth general points of $X$ such that $[x_{1}+\dots+ x_{s}]$ is a smooth general point of $\sigma_{s}(X)$. Then
\[
\widehat{T}_{x_{1}+\dots+ x_{s}}\sigma_{s}(X) =  \{ \widehat{T}_{x_{1}}X,\dots \widehat{T}_{x_{s}}X  \}
.\]
\end{lemma}

 There is a one-to-one correspondence between nonzero $k$-vectors $e_I:= e_{i_{1}}\wedge e_{i_{2}}\wedge\dots\wedge e_{i_{k}}$ and square-free monomials $e_{i_{1}}\dots e_{i_{k}}$, so we often omit the $\wedge$ symbols.  For shorthand, we write $\widehat{T}_{i_1,\ldots, i_k} :=\widehat{T_{e_{i_1},\ldots, e_{i_k}}\Gr(k,V)}$.

\subsection{The case of \texorpdfstring{$2$}{}-planes}\label{sec:2planes}
Asking for too much overlap causes collapsing, such as the following.
\begin{prop}
Suppose $r \geq k-1$ and $\dim(V)  = n\geq k.$ Then,
\begin{equation}
\sigma_s^{r}(\Gr(k,n)) = \Gr(k,n).
\end{equation}
\end{prop}
\begin{proof}
When $r \geq k$ then the proof is trivial.
Now consider the case $r= k-1$. 
An open subset of points in the cone $\widehat {\sigma_s^{k-1}(\Gr(k,V))}$ can be written as 
\[v_1\cdots v_k + v_1\cdots v_{k-1}v_{k+1}+ \cdots +  v_1\cdots v_{k-1}v_{k+s-1},\]
for  $v_i \in V$. This expression factors as 
\[v_1\cdots v_{k-1}v_{k+1} (v_k+\cdots v_{k+s-1}),\]
which is clearly an element of $\widehat {\Gr(k,V)}$. So, the result follows.
\end{proof}

\subsection{The case of \texorpdfstring{$3$}{}-planes}
The first non-trivial case of restricted secant varieties is that of the $1$-restricted chordal variety of $\Gr(3,V)$, with $\dim (V) \geq 5$.

\begin{prop} \label{prop:3case}
Consider $X = \Gr(3,V)$ with $\dim(V) =n\geq 5$. Then the following hold:
\begin{enumerate}
    \item $\widehat{T}_{e_{1} e_{2} e_{3} +e_{1} e_{4} e_{5} }\sigma_{2}^{1}(X) = 
	 V \cdot  \{ e_{2} e_{3} + e_{4} e_{5} ,e_{1} e_{2},e_{1} e_{3} ,e_{1} e_{4}, e_{1} e_{5}  \}
	 $, 
	 \item $\widehat{T}_{e_{1} e_{2} e_{3} + e_{4} e_{5} e_{6} }\sigma_{2}(X) = 
\widehat{T}_{e_{1} e_{2} e_{3} +e_{1} e_{4} e_{5} }\sigma_{2}^{1}(X)  \oplus \{e_{i}e_{2}e_{3}-e_{i}e_{4}e_{5} \mid i\geq 6  \} \backslash ( \widehat{T}_{123}\cap \widehat{T}_{145})$,
\item and $\dim ( \sigma_2^1 (X)) = 5n-16 $ and $\sigma_2^1 (X)$ has codimension $n-1$ in $\sigma_2 (X)$. 
\end{enumerate}
\end{prop}

\begin{remark}
Note  that  $\widehat{T}_{123}\cap \widehat{T}_{145} = e_{1}\cdot \{ e_3 e_4,  e_2 e_4, e_3 e_5, e_2 e_5  \} = \CC^{4}$
and that $\dim ( \sigma_2^1 (X)) = 5n-16 = \dim (\sigma_2 (X)) - 4 - (n-5)$, where the ``4'' in the right-hand side signifies an extra intersection. It is also notable that the monomials in (1) correspond to the triangles in a square triangulated by adding a central point and all edges to that point.
\end{remark}

\begin{proof}
We mimic how one would prove Terracini's lemma. First we recall how to compute the cone over the tangent space to the Grassmannian.
Use  Def.~\ref{def:tgt} and construct a curve $\gamma(t)=e_{1}(t) e_{2}(t) e_{3}(t)$ such that $e_{i}(0)=e_{i}$ and let  $e_{i}'$ denote $e_{i}'(0)$, for $1\leq i\leq 3$. Then 
\[
\gamma(t)'_{| t=0} = e_{1}' e_{2} e_{3} + e_{1} e_{2}' e_{3} + e_{1} e_{2} e_{3}'.
\]
Since the vectors $e_{i}'$ are arbitrary in $V$, 
\[
\widehat{T}_{123} = V\cdot \{e_{1} e_{2},e_{1} e_{3} ,e_{2} e_{3}\} \cong \{e_1e_2e_3\} \oplus  \left(V/\{e_{1},e_{2},e_{3}\} \right) \cdot\{e_{1} e_{2},e_{1} e_{3} ,e_{2} e_{3} \} \cong \CC^{(n-3)3+1}
,\]
which agrees with the description given in \cite{CGG6_Grassmann}*{p~638}. The spaces $\widehat T_{145}$ and $\widehat T_{456}$ are similarly defined.
Now let $\gamma(t)=e_{1}(t) e_{2}(t) e_{3}(t) + e_{1}(t) e_{4}(t) e_{5}(t) = e_{1}(t) (e_{2}(t) e_{3}(t) +e_{4}(t) e_{5}(t))$. Then
\[
\gamma(t)'_{| t=0} = e_{1}'( e_{2} e_{3} +e_{4} e_{5}) + e_{1} e_{2}' e_{3} + e_{1} e_{2} e_{3}' +  e_{1} e_{4}' e_{5} + e_{1} e_{4} e_{5}'
,\] with $e_{i}(0)=e_{i}$ and 
$e_{i}'(0) = e_{i}'$  for $1\leq i\leq 5$.
Since $e_{i}'$ are arbitrary in $V$, we arrive at (1). 

For (2), note that by a similar calculation we have
\[
\widehat{T}_{e_{1} e_{2} e_{3} +e_{4} e_{5} e_{6} }\sigma_{2}(X) = 
	 V \cdot  \{ e_{1} e_{2}, e_{1} e_{3} ,e_{2} e_{3},e_{4} e_{5} ,e_{4} e_{6}, e_{5} e_{6}  \}.
\]
Compare the tangent spaces $\widehat{T}_{e_{1} e_{2} e_{3} +e_{1} e_{4} e_{5} }\sigma_{2}^{1}(X)$ and $\widehat{T}_{e_{1} e_{2} e_{3} +e_{4} e_{5} e_{6} }\sigma_{2}(X)$. The required overlap on $\sigma_{2}^1(X)$ forces elements of the form $\{e_{i}e_{2}e_{3}-e_{i}e_{4}e_{5} \mid i\geq 6  \}$ to be excluded. Therefore, 
\[
\widehat{T}_{e_{1} e_{2} e_{3} + e_{4} e_{5} e_{6} }\sigma_{2}(X) = 
\widehat{T}_{e_{1} e_{2} e_{3} +e_{1} e_{4} e_{5} }\sigma_{2}^{1}(X)  \oplus \{e_{i}e_{2}e_{3}-e_{i}e_{4}e_{5} \mid i\geq 6  \}.
\]

Now we prove (3). In the case $n=5$, one shows that a general point on $\sigma_{2}(\Gr(3,5))$ can be written (after a possible change of basis) as $[e_{1} e_{2} e_{3} + e_{1} e_{4} e_{5}]$, which is in $\sigma_{2}^{1}(\Gr(3,5))$, hence $\sigma_{2} (\Gr(3,5))=\sigma_{2}^{1} (\Gr(3,5))$. 

For $n\geq 6$, a general point of $\sigma_{2}(X)$ is (up to a change of basis) $e_{1} e_{2} e_{3} + e_{4} e_{5} e_{6}$. Therefore, we obtain orbit stability, and by Proposition~\ref{prop:dimfam} we get the dimension count $6k-17$.

So, by Terracini's lemma and Grassmann's formula
\[\widehat{T}_{e_{1} e_{2} e_{3} + e_{4} e_{5} e_{6}}\sigma_{2}(X)
=\widehat{T}_{123}+\widehat{T}_{456} = \{\widehat{T}_{123}\cup \widehat{T}_{456} \}- \{\widehat{T}_{123}\cap \widehat{T}_{456} \}=\{\widehat{T}_{123}\cup \widehat{T}_{456} \}
.\]

Now compare the two tangent spaces, and notice that 
\begin{multline}\label{eq:123456}
\widehat{T}_{e_{1} e_{2} e_{3} + e_{4} e_{5} e_{6}}\sigma_{2}(X)
=
\widehat{T}_{e_{1} e_{2} e_{3} + e_{1} e_{4} e_{5}}\sigma_{2}^{1}(X)\oplus
\{ e_i e_2 e_3 - e_i e_4 e_5 \mid i\geq 6  \}
\\
=\widehat{T}_{123}\oplus \widehat{T}_{456} = (\widehat{T}_{123}+\widehat{T}_{145}) \oplus
 \{ e_i e_2 e_3 - e_i e_4 e_5 \mid i\geq 6  \}
\\
=\widehat{T}_{123}\oplus \widehat{T}_{456} = ((\widehat{T}_{123}\cup \widehat{T}_{145})-(\widehat{T}_{123}\cap(\widehat{T}_{145}))) \oplus
 \{ e_i e_2 e_3 - e_i e_4 e_5 \mid i\geq 6  \}
.\end{multline}
So the formula for (3) follows by noting that the ``$-4$'' comes from (1) and the  ``$-(n-5)$'' comes from the complement on the right hand side of \eqref{eq:123456}. 
\end{proof}

Examples like these and computations done in \texttt{M2} led to the generalizations in Section~\ref{sec:rrestricted}.

\section{Expected Dimensions for \texorpdfstring{$r$}{}-restricted secant varieties}\label{sec:rrestricted}
Recall for varieties $X,Y \subset \PP V$ the \defi{abstract join} variety is 
\[
J(X,Y) = \overline{\left\{([x],[y],[p]) \mid p \in \text{span}\{x,y\} \right\}} \subset \PP V \times \PP V \times \PP V,
\]
where the overline denotes Zariski closure. The \defi{abstract $s$-secant variety} of $X$ is denoted $\Sigma_s(X)\subset (X)^{\times s}\times \PP V$ and can be constructed inductively as the $s$-fold join of $X$ with itself:
\[
\Sigma_s(X) = \overline{\left\{([x_1],[x_2],\dots,[x_s],[p]) \mid p \in \text{span}\{x_1,\dots,x_s\} \right\}} \subset \PP V^{\times s} \times \PP V.
\]
The \defi{embedded $s$-secant variety} is the projection to the last factor,  denoted $\sigma_s(X) \subset \PP V$.
The \defi{virtual dimension} of the $s$-secant variety is the dimension of the abstract $s$-secant variety:
\[
\text{v.dim}(\sigma_s(X)) = \dim (\Sigma_s(X) )=  s\cdot \dim(X)+s-1.
\]
The \defi{expected dimension} of $\sigma_s(X)$  is 
\[\exp.\dim (\sigma_s(X)) = 
\min\{\dim (\PP  V), \dim (\Sigma_s(X)) \}= \min\{\dim (\PP  V),  s\dim (X) +s-1\}. 
\]
Similarly, the abstract $r$-restricted $s$-secant variety is the incidence variety
\[
\mathcal{I} \subset \Gr(r,V) \times \Gr(k-r,V)^{\times s}\times \PP \bw{k} V,
\]
defined by 
\[\mathcal{I}:= \overline{\{(E, F_1,\ldots, F_s, [z]) \mid z\in  \text{span}\{ \widehat E\wedge \widehat F_1, \ldots, \widehat E \wedge \widehat F_s \} \}}.\] 
This incidence variety is natural as it mimics the way one might choose a point in $\sigma_s^r(\Gr(k,n))$. That is, select an $r$-plane for the overlap, then select the $s$  $(k-r)$-planes in the complement. Finally, select the $(s-1)$ points needed to define the secant variety.
However, what we say is ``expected'' should change based on how many $k$-planes we are trying to fit into a vector space $V$ with an $r$-dimensional overlap, and we handle this in several cases.

As the restricted secant variety depends on the intersection of $s$ linear spaces, Grassmann's formula calculates the size of the intersection of exactly two vector spaces. We apply this to the case of the restricted chordal variety below, where let $E \in \Gr(r,V)$ and $V/E$ to respectively denote the $r$-dimensional space and its quotient. 
\begin{remark}
Recall that the set of skew-symmetric matrices of rank $\leq r$ corresponds to the secant variety $\sigma_r (\Gr(2,V))$,  which is always defective.
\end{remark}

\begin{prop} \label{prop:2d-defect}
Let $n=k+2$ and $r =\max(r,2k-n)$. Then,
\[\edim(\sigma_2^{r}(\Gr(k,n)))=
\min\left\{\binom{n}{k} -1,r(n-r)+2((k-r)(n-k))+1\right\},
\] 

\[
\vdim(\sigma_2^{r}(\Gr(k,n)))=
\min\left\{\binom{n}{k} -1,r(n-r)+2((k-r)(n-k))-3\right\}. 
\]
Further $\sigma_2^{r+1}(\Gr(k,n)))=\Gr(k,n)$.
\end{prop}
\begin{proof}
Let $n=k+2$. This case handles spaces that have greater than a one-dimensional overlap $(2k-(k+2)=k-2)$. The corresponding incidence variety is composed of a secant of Grassmannian of lines $\sigma_{2}(Gr(2,V))$, which are known to be defective.  Redefine, if necessary, $r :=\max(r,2k-n)$. 
An isomorphic incidence variety to the one given in the proposition above has the form 
\[ \mathcal{I} \subset \Gr(r,V) \times \Gr(k-r,V/E)^{\times 2}\times \PP \bw{k} V.\] \

Let $N = \binom{n}{k}-1.$ The expected dimension is \\
\begin{equation}
\begin{aligned}
\edim(\sigma_2^{r} (\Gr(k,V)) ):= & \min\{ \dim(\mathcal{I}), N\} \\
= & \min\{ r(n-r) +(k-r)(n-(k-r))+1, N\}.
\end{aligned}
\end{equation}

Consider another representation of the same restricted chordal variety in the form \[ \mathcal{I} \subset \Gr(r,V) \times \sigma_2(\Gr(k-r,n-(k-r))).\] 

All secant varieties of lines are defective \cite{CGG6_Grassmann}. Therefore, as these restricted chordal varieties are composed of a Grassmannian and a point in $\PP \bw{k} V$ which have full dimension and one piece that is defective, namely the secant variety of lines, the restricted chordal variety is defective. 
Then, by direct calculation the actual dimension is the expected dimension minus one copy of $\Gr(k-r,n-(k-r))$. \
\end{proof}
\begin{example}
Consider $\sigma_2^1(\Gr(6,8))$. By Grassmann's formula any pair of $6$-dimensional subspaces of an 8-dimensional space has at least a $4$-dimensional intersection. Therefore $\sigma_2^{i}(\Gr(6,8))$ with $1 \leq i \leq 4$ are all equal. Those varieties have  dimension $21$ but from the incidence description the expected dimension would be $16+(4+4+1)=25$. This is exactly the known defect for $\sigma_s(Gr(2,V)$ which is $2s(s-1)$ or $4$ when $s=2$ \cite{CGG6_Grassmann}.
\end{example}

\begin{prop}
Let $k=r+2$. Then, $\sigma_2^{r}(\Gr(k,n))$ is defective, with defect
$2s(s-1)$.
\end{prop}

\begin{proof}
 Let $k=r+2$. Then, $\sigma_2^{r}(\Gr(k,n))$ is defective. Here the removal of the $r$-dimensional overlap leaves $\sigma_2(\Gr(k-r,V/E))$ and $k-r=2$ meaning it is also a secant variety of a Grassmannian of lines which is known to be defective. Construct the incidence variety as follows:
\[ \mathcal{I} \subset \Gr(r,V) \times \sigma_2(\Gr(k-r,V/E)).\]
 The incidence variety for this restricted chordal variety is also composed of a secant variety of lines which we know to be defective. The expected and virtual dimension counts are  then exactly the same as \ref{prop:2d-defect}, however $\sigma_2^{r+1}(\Gr(k,n)))\neq\Gr(k,n)$.
\end{proof}
\begin{prop}
Suppose $2k-1 \le n \leq k+2$ and $r =\max(r,2k-n)$. Then the virtual and expected dimensions for $\sigma_2^{r}(\Gr(k,n))$ are:
\[
\vdim(\sigma_2^{r} (\Gr(k,V)) ):= \dim(\mathcal{I}) = r(n-r) + 2 (k-r)(n-k) +1
,\]
and
\[
\edim(\sigma_2^{r}(\Gr(k,V)))= \min\left\{ \vdim(\sigma_2^{r} (\Gr(k,V)) ), \binom{n}{k} -1 \right\}.\]
\end{prop}
\begin{proof}
When $n \geq 2k-1$, count parameters in the following manner. First, choose $E\in \Gr(r,V)$, then choose $ F_1,F_2\in \Gr(k-r,V/E)$, and finally $z$ on the line $ \{ \widehat E \wedge \widehat F_1, \widehat E \wedge \widehat F_2 \} $. This gives the dimension counts listed in the proposition.
\end{proof}

The $r$-restricted chordal variety may also be defined as the following orbit closure
\[
\sigma^{r}_{2}(\Gr(k,V)) := \overline{
\GL(V).[
e_{1} e_{2} \dots e_{r}
 (e_{r+1} \dots e_{k} + e_{k+1} \dots e_{2k-r} )] }
,\]
which is equivalent to Def.~\ref{def:rrest}.
The dimension of $\sigma^{r}_{2}(\Gr(k,V)) $ is the dimension of the tangent space at a general point (i.e., on the orbit).

\begin{remark}
We also have a nice description of the tangent space of the restricted chordal variety using $\mathcal{I}$, that is it is the image of the tangent space to $\mathcal{I}$ under the projection:
\[
\widehat{T}_{E}\Gr(r,V) \times (\widehat{T}_{A}\Gr(k-r, V) + \widehat{T}_{B}\Gr(k-r, V)) \subset \bw{r} V \times \bw{k-r} V \subset \bw{r} V \otimes \bw{k-r} V 
\]
\[
\downarrow \pi
\]
\[
\bw{k}V
\]
\end{remark}

\section{Dimensions for  \texorpdfstring{$r$}{}-restricted secant varieties}\label{sec:dimr}

It turns out that restricted secant varieties are birational to a fiber bundle, which can, in turn, be used to understand their dimension. It may be possible to further exploit this connection like what was done in \cite{Landsberg-Weyman-Bull07}, which applied Weyman's Geometric Technique to a similar partial desingularization to obtain generators of the ideal.

\fiber*

\begin{proof}
Let $\Xi$ denote the fiber bundle in the statement of the theorem.
Recall the tautological sequence of bundles over the Grassmannian $\Gr(r,V)$:
\[
\xymatrix{
0 \ar[r]& \mathcal{S} \ar[r] & \underline{V} \ar[r]  & \mathcal{Q} \ar[r]& 0  
 }\]
where over a point $E\in \Gr(r,V)$ the fiber of the subspace bundle $\mathcal{S}$ is $E$, the fiber of the trivial bundle $\underline V$ is $V$ and the fiber of $\mathcal{Q}$ is $V/E$. Applying  the Schur functor $\bw{k-r}$ we obtain a vector bundle:
\[
\xymatrix{
\bw{k-r} Q \ar[d] \\ 
  \Gr(r,V)
}
\]
whose fiber over $E$ is $\bw{k-r}(V/E)$. In each fiber we have (a copy of) $\sigma_s(\Gr(k-r, V/E))$. We depict this in the following diagram.
\[
\xymatrix{
\sigma_s(\Gr(k-r , V/E)) \ar@{^{(}->}[r] \ar[dr]& \PP  \bw{k-r}V/E  \ar@{^{(}->}[r]  & \PP  \bw{k-r} Q \ar[d] \\ &  E \in &  \Gr(r,V)
}\]
The total space of the fiber bundle $\Xi$ consists of pairs $(E, [t])$ with $[t]\in \sigma_s(\Gr(k-r,V/E))$, and on an open subset we can assume that $t$ has rank at most $k$ (not just border rank $k$).
Select such a pair $(E, [t])$. 
For $E \in \Gr(r,V) \subset \PP \bw r V$ we write $E = [e_1\wedge \cdots \wedge e_r] $ for independent elements $e_i \in V$. Elements in an open subset of $ \sigma_s(\Gr(k-r,V/E))$ are of the form $[t] = [t^{(1)} + \cdots + t^{(k)}] $, with $[t^{(i)}]  = [a_1^{(i)}\wedge \cdots \wedge a_{k-r}^{(i)}] \in \Gr(k-r,V/E)$ for each $i$. 

Define a rational map $\Phi \colon \Xi  \dashrightarrow \sigma_s^r\Gr(k,V) $ via
\[
\Phi(E, [t]) = [e_1\wedge \cdots \wedge e_r\wedge t]
\]
on the open subset of points $(E,[t])$ in $\Xi$ such that $e_1\wedge \cdots \wedge e_r\wedge t$ is non-zero and $\rank t \leq k-r$.

The image is indeed in $\sigma_s^r(\Gr(k,V))$ since the collection $(\widehat E\wedge t^{(1)},\ldots, \widehat E\wedge t^{(s)})$  is a set of forms representing $k$-planes with (at least) an $r$-dimensional intersection.
This mapping is dominant because an open subset of points of $\sigma_s^r\Gr(k,V)$ have a representation as $[\widehat E\wedge t]$.

Now we describe a rational map $\Psi \colon \sigma_s^r(\Gr(k,V)) \dashrightarrow \Xi$. Choose a basis $\{v_1,\ldots,v_n\}$ of $V$ and  volume form $\Omega_V := v_1 \wedge \cdots \wedge v_n \in \bw n V$. This induces isomorphisms $\bw j V \to \bw {n-j} V^*$ via contraction (Hodge star) with $\Omega_V$. This mapping is graded in the following sense.
\begin{lemma}\label{lem:hodge}
Suppose $A,B$ are respectively vector spaces of dimensions $a,b$, and let $A\oplus B$ denote their external direct sum.
Let $\alpha \in \bw{i} A$ and $\beta \in \bw j B$.  Then $\alpha \wedge \beta \in \bw{i+j} (A\oplus B)$. Moreover,
\[
\Omega_{A\oplus B}(\alpha \wedge \beta) = 
(-1)^{i+j}
\Omega_A(\alpha) \wedge \Omega_B(\beta) 
,\]
 in $\bw{a-i} A^* \otimes \bw{b-j}B^* \subset \bw{a+b-(i+j)}(A \oplus B)^*$.
\end{lemma}
\begin{proof} Since the mappings $\Omega_{A\oplus B}, \Omega_A, \Omega_B$ are all linear, it suffices to prove the statement on rank-one elements,  $\alpha=a_1 \wedge \cdots \wedge a_i$ and $\beta= b_1 \wedge \cdots \wedge b_j$.
We may choose an adapted basis $\{a_1,\ldots, a_a, b_1,\ldots,b_b\}$ of $A\oplus B$ so that the first $a$ vectors come from $A$ and the next $b$ vectors come from $B$. Moreover, we can select the first $i$ vectors from the terms of $\alpha$, and extend to a basis of $A$ to obtain the  next $a-i$ vectors. Similarly, for the last $b$ be choose a basis of $B$ starting from the terms of $\beta$. We also choose a dual basis $\{a^1\cdots a^a,b^1 \cdots b^b\}$ of $(A\oplus B)^*$. 
 Now apply the contraction operator to  $\alpha \wedge \beta=a_1\wedge\cdots \wedge a_i \wedge b_1\wedge \cdots \wedge b_j$:
\[
\Omega_{A\oplus B}(\alpha \wedge \beta)=(-1)^{i+j} \times a^1\wedge\cdots \wedge a^{a-i} \wedge b^1\wedge \cdots \wedge b^{b-j}.  
\] 
where $(-1)^{i+j}$ defines the sign of the permutation that passes the $a_i$'s through the $b_j$'s to get it in the form $a^1\wedge\cdots \wedge a^{a-i} \wedge b^1\wedge \cdots \wedge b^{b-j}.$
Then, as $\Omega_{A}(\alpha)=a^1\wedge\cdots \wedge a^{a-i}$ and $\Omega_B(\beta)=b^1\wedge \cdots \wedge b^{b-j}$, substituting into the right-hand side yields:
\[
\Omega_{A\oplus B}(\alpha \wedge \beta) = 
(-1)^{i+j}
\Omega_A(\alpha) \wedge \Omega_B(\beta) 
.\] One checks that the result is independent of the choice of bases of $A$ and $ B$.
\end{proof}

Now let $[w] \in \sigma_s^r(\Gr(k,n))$ be a general point, so that 
\[
w = \sum_{i=1}^s e_1^{(i)}\wedge\cdots \wedge e_k^{(i)},
\]
with $E_i = [e_1^{(i)}\wedge\cdots \wedge e_k^{(i)}] \in \Gr(k,n)$ for each $i$, and with $\cap_i E_i = E$ an $r$-dimensional subspace of $V$. 
More explicitly, let $\pi$ denote the projection from the abstract secant variety. 
General points are selected from the complement of the following closed subset:
\[\{\pi(E_1, \ldots, E_s, [w]) \mid \rank(E_i)< k \text{ for some }i  \text{ or } \dim (\cap_i E_i) <r  \}.\]
We wish to find an expression (after a possible change of basis) like
\[w =
    e_1\wedge \cdots \wedge e_r\wedge(a^{(1)}_{1}\wedge \cdots \wedge a^{(1)}_{k-r}) +
        \cdots 
   + e_1\wedge \cdots \wedge e_r\wedge(a^{(s)}_{1}\wedge \cdots \wedge a^{(s)}_{k-r}) 
,\]
which factors as 
\[w =
    e_1\wedge \cdots \wedge e_r\wedge\left(a^{(1)}_{1}\wedge \cdots \wedge a^{(1)}_{k-r} +
        \cdots 
   +a^{(s)}_{1}\wedge \cdots \wedge a^{(s)}_{k-r} \right) 
,\]
and hence can be readily seen to be an element in $\bw r E \otimes \bw{k-r} V/E$. 
If we can do this, then the mapping from such a point to $\Xi$ will be clear. 

Apply $\Omega_V$ to this expression for $w$ to obtain (via  Lemma~\ref{lem:hodge})
\[\Omega_V(w) =
    \Omega_{E}(e_1\wedge \cdots \wedge e_r)\cdot \Omega_{V/E}(\left(a^{(1)}_{1}\wedge \cdots \wedge a^{(1)}_{k-r} +
        \cdots 
   +a^{(s)}_{1}\wedge \cdots \wedge a^{(s)}_{k-r} \right) 
.\]
 We can take the scalar factor $\Omega_{E}(e_1\wedge \cdots \wedge e_r) $ to be equal to $1$  so that
\[\Omega_V(w) =
    \Omega_{V/E} \left(a^{(1)}_{1}\wedge \cdots \wedge a^{(1)}_{k-r} +
        \cdots 
   +a^{(s)}_{1}\wedge \cdots \wedge a^{(s)}_{k-r} \right) 
,\]
and by construction the summands in $\Omega_V(w)$ live in $\bw{n-r} V/E$.  Moreover, 
\[
[\Omega_V(w)] \in \sigma_s(\Gr(n-r, V/E)).
\]
Note that $\Omega_V(w) \in \bw{n-r} V/E$ in particular.
Consequently, one can find $E$ from $\Omega_V (w)$ as the annihilator in the dual of the kernel of the 1-flattening defined for  $T \in \bw{n-r} V^*$ as
\[
F_T \colon V \to \bw{n-r-1}V^*
\]
applied to $T = \Omega_V(w)$. Once $E = \ker F_{\Omega_V(w)}$ is found, one can find an expression for $[t] \in \sigma_k(\Gr(n-k, V/E))$ by applying the projection operator $\Omega_{V/E}$ to $\Omega_{V}(w)$.

This process gives a method for producing from $[w] \in \sigma^r_s\Gr(k,V)$ a pair $(E, [t]) \in \Xi$. In particular $\Psi([w]) \mapsto (E,[\Omega_{V/E}(\Omega_V(w))]) $, with $E = \ker F_{\Omega_V(w)}$. 
By construction the composition of these two mappings is the identity on the open sets where they are defined.
\end{proof}

The description of the restricted secant varieties suggests the following regarding the minimal defining equations of the ideals of secants of restricted secant varieties, which was studied in the case of usual secants by one of us \cite{daleo2016computations}.
\begin{conj}
Consider $X = \sigma_s^r(\Gr(k,n))$ with parameters $s,r,k,n$ so that $X$ is non-trivial. Then the ideal of $X$ is generated by two types of polynomials:
\begin{enumerate}
    \item polynomials inherited from the ideal of $\sigma_s(\Gr(k-r,n-r) )$, i.e. the polynomials coming from the condition that $\Omega(w) \in \sigma_s(\Gr(k-r,n-r) )$ for $w \in \sigma_s^r(\Gr(k,n))$.
    \item polynomials coming from the the $(r+1) \times (r+1)$ minors of the 1-flattening $F_T \colon V \to \bw{n-r-1}$ for $T = \Omega(w)$. 
\end{enumerate}
\end{conj}

A conjecturally complete list (from \cite{BaurDraismadeGraaf} ) of known defective secant varieties of Grassmannians can be found at  Table~\ref{tab:defective}. We can combine the considerations above with the BDdG-Conjecture \cite{BaurDraismadeGraaf} to say that the defectivity of $r$-restricted higher order secant varieties only depends on the usual notion of $k$-defectivity of secant of Grassmannians.

\classification*

\begin{proof}
Let $\sigma_s^{r}(\Gr(k,V))$ be the $r$-restricted $s$-secant variety and define the corresponding incidence variety $\mathcal{I} \subset \Gr(r, V) \times \sigma_s(\Gr(k-r, V/E))$. 
We showed in Theorem~\ref{thm:fiber} that the restricted secant is bi-rational to this incidence variety, and its dimension is completely determined by the dimension of the usual secant variety.
 Therefore, any defect must come from $\sigma_s(\Gr(k-r, V/E))$. The current list of known defective cases are exactly those in the BDdG conjecture.
\end{proof}
\begin{table}[t!]
\captionsetup{width=\textwidth}

\renewcommand{\arraystretch}{1.1}
\centering
 \begin{tabular}{||c |c| c| c||} 
 \hline
   Secant Variety & actual codimension & expected codimension \\ [0.75ex] 
 \hline\hline
 $\sigma_s(\Gr(2,n))$ & $2s(s-1)$ & 0 \\
$\sigma_3(\Gr(3,7))$ & 1 & 0 \\
$\sigma_3(\Gr(4,8))$ & 20 & 19 \\
$\sigma_4(\Gr(4,8))$ & 6 & 2 \\
$\sigma_4(\Gr(3,9))$ & 10 & 8 \\
 [0.5ex] 
 \hline
 \end{tabular}
\caption{The conjecturally complete list of defectivity for secants of Grassmannians \cite{BaurDraismadeGraaf}.
}\label{tab:defective}
\end{table}

The following is the special case of Corollary~\ref{thm:classification} for $r$-restricted chordal variety.

\begin{prop}
\label{prop:chordaldim}
The projection from the incidence variety 
\[\mathcal{I} \subset \Gr(r, V) \times \sigma_2(\Gr(k-r, V/E)) \to \PP( \bw k V),\] whose image is $\sigma_2^{r}(\Gr(k,V))$, has finite fibers. 
Hence given the BDdG conjecture $\sigma_2^{r}(\Gr(k,V))$ has no additional defect other than the defect coming from  (usual) secant varieties of Grassmannians.
The only defective restricted chordal varieties of Grassmannians are when $n=k+2$ or when $k-r = 2$. 
\end{prop}

We confirmed this statement for those $r$-restricted chordal varieties composed of $\Gr(2,n)$ for several examples in Macaulay2. We also calculated the dimension for several other known cases. For example, $\sigma_3^{1}(\Gr(4,8))$ which is composed of $\sigma_3^1 (\Gr(3,7))$ has dimension $40$, however its expected dimension is $45$ indicating it is in fact defective. We also performed similar checks of other $r$-restricted chordal varieties composed of a defective secant variety.

\section{Coding Theory}\label{sec:coding}
Let us recall several relevant coding theory definitions from \cite{hankerson2000coding}. 
Let $F$ denote an \defi{alphabet}, which is a set of digits. A sequence of digits from $F$ is called a \defi{codeword}. The \defi{length} of a codeword is the number of digits in the codeword. The collection of codewords, denoted $C$, is called  a \defi{dictionary}. A \defi{code} of length $n$ is a collection of codewords. 
A code is called a \defi{binary code} if $F = \{0,1\}$. A code is transmitted by sending the digits of its codewords in sequence across a channel. The \defi{Hamming distance} between two codewords of equal length  $u,v \in C$, denoted $d(u,v)$, is the number of places that $u$ and $v$ differ. For a codeword $u$, the \defi{weight} of $u$ is defined as, $w(u)=d(u,0)$ where $0$ corresponds to the $0$ digit in the given alphabet. 
Abo-Ottaviani-Peterson gave the following connection to geometry.
\begin{theorem}{\cite{AOP_Grassmann}*{Theorem~4.1} }
Let $A(n, 6, w)$ be the cardinality of the largest binary code of length $n$, constant weight $w$, and Hamming distance between any two codewords at least $6$. 
If $s \leq A(n + 1, 6, k + 1)$ then $\sigma_s(\Gr(k, n))$ has the expected dimension. 
\end{theorem}

A \defi{Grassmann code} is a special case of a linear code. Let $\mathbb{F}_q$ be the field with $q$ elements. Then, it is well-known that $\Gr_{\mathbb F _{q}}(k,n)$ contains $P$ points where
\begin{equation}\label{eq:numGr}
P=\frac{(q^n-1)(q^{n-1}-1)\dots(q^{n-k+1}-1)}{{(q^k-1)(q^{k-1}-1)\dots(q-1)}}
\end{equation}

To define the Grassmann code as a linear code first pick a Pl\"ucker representative of each of the $P$ points as a column vector in $(\mathbb F _q)^I$ for $I=\binom{n}{k}$ and form an $I \times P$ matrix $M$ (the generator matrix) with these $P$ vectors as columns. Grassmann codes (in the identifiable case) correspond to sums of $k$-fold wedge products.
Vectors in the Pl\"ucker embedding of $\Gr(k,n)$ are the codewords in a Grassmann code. So a general $x \in \sigma_s(\Gr(k,n))$ can be thought of as an unordered collection of $s$ codewords. The codewords are uniquely recoverable as long as $\Gr(k,n)$ is \defi{identifiable} in rank $s$, which we expect is true for small $s$ \cite{casarotti2022tangential}. 

The \defi{Grassmannian distance} for $A,B \in \Gr(k,n)$ is  $d_G(A,B)=k-dim(A \cap B)$. Note, points of the restricted chordal variety $\sigma_2^r(\Gr(k,n))$ are of the form $[\hat A+\hat B]$, with $d_G(A,B) = k-r$. 

A code corresponding to a point of $\sigma_s^{r}(\Gr(k,n))$ (again assuming identifiability), consists of a collection of $s$ codewords with the restriction that (pairwise) codewords must have distance $k-r$ between them, and that the intersection is the same for all pairs. This leads to a trade-off between redundancy and the capacity of the coding scheme. The restriction limits the number of possible  codewords available, corresponding to an increase in the amount of information necessary to ensure accurate decoding. The max number of codewords in a signal for a given coding scheme can be considered the capacity of the channel, which is, in turn, found by determining the dimension of the variety (i.e. $\dim(\sigma_s^r(\Gr(k,n)))$ and $ \dim(\sigma_s(\Gr(k,n)))$) corresponding to the coding scheme.

Section~\ref{sec:dimr} provides a method  involving the contraction operator to determine whether a given point lies on a restricted chordal variety. The contraction determines the common intersection and the remaining information could be computed separately by tensor decomposition. Therefore, with an appropriate choice of collections of codewords on restricted secants one could build  extra information for decoding as redundancies in each codeword. This redundancy could permit an error-correcting mechanism. 

Theorem \ref{thm:fiber} says the following in terms of the coding theory. Codes for restricted secants of Grassmannians can be thought of as Grassmann codes except that the codewords are padded with an additional overlap. Therefore, \cite{AOP_Grassmann}*{Theorem~4.1} says:  Let $A(n, 6, w)$ be the cardinality of the largest binary code of length $n$, constant weight $w$, and distance $6$. 
If $s \leq A(n + 1, 6, k + 1)$ then $\sigma_s^r(\Gr(k+r, n+r))$  has the expected dimension. 

We end this section with an extended example.
\begin{example}
Consider binary codes in the case of $\Gr(3,\FF_2^6) \subset \PP \bw{3} \FF_2^6$.  
By \eqref{eq:numGr} there are $1,395$ points in $\Gr(3,\FF _{2}^6)$. The corresponding linear code has a $20 \times 1,395$ generator matrix, $M$, whose columns are the Pl\"ucker coordinates of each of the $1,395$ points. Then, one encodes a message $b$ as the product $Mb$.

Special subsets of possible messages come from points of a given orbit (like the secant or restricted secant, or tangent to the Grassmannian). For a variety $X$ ``the orbit'' is the set, denoted $X^\circ$, of points that are equivalent to the normal form on the respective variety up to change of coordinates by $\SL_6(\FF_2)$. We are interested in the numbers of points in each orbit.

For a pair of codewords $x,y \in \Gr(3,6)$, construct the message $b$ consisting of two non-zero entries. This represents a code in $\sigma_2(Gr(3,6))$. Changing the codewords $x,y \in \Gr(3,6)$ so that they share an $r$-dimensional overlap results in a message in $\sigma_2^r(\Gr(3,6)).$ 

Here we can completely describe the $\SL_6(\FF_2)$-orbits in $\bw 3 \FF_2^6$.
To count the number of points in an orbit of a finite matrix group we repeatedly apply random non-singular matrices to the set of known points in the orbit until the number of unique elements in the set stabilizes. This indicates that it is likely that all the points in that orbit have been obtained. 
If the list of orbits obtained this way fills out the entire ambient space we are ensured that no points were missed. On the other hand, if there are missing points one can take the orbit of a point not already on a known orbit, and compute its orbit. The results are listed in Table~\ref{tab:numOrbits}.

\begin{table}[h]
    \centering
\begin{tabular}{|r||c|c|c|c|c|c|}
\hline
$X^\circ$  & $0$ & $\Gr(3,6)^\circ$ & $\sigma_2^1 (\Gr(3,6))^\circ$ & 
$\tau (\Gr(3,6))^\circ$ & $\sigma_2 (\Gr(3,6))^\circ$ & $\Xi^\circ$ \\
\hline 
$\# X^\circ $ & 1 & 1,395 & 54,684 & 468,720 & 357,120 & 166,656\\
\hline
\end{tabular}
    \caption{The orbits of $\bw{3} \FF_2^6$ under the $\SL_6(\FF_2)$-action.}
    \label{tab:numOrbits}
\end{table}

The classical orbit closures are linearly ordered: $\Gr(3,6) \subset \sigma_2^1 \Gr(3,6) \subset \tau(\Gr(3,6)) \subset \sigma_2(\Gr(3,6)) = \PP \bw 3 \FF_2^6 $.
We found precisely one new orbit, with normal form:
\[
\xi = e_1e_2e_4+e_0e_3e_4+e_0e_2e_5+e_0e_3e_5+e_1e_3e_5
 = (e_1e_2+e_0e_3)e_4+(e_0e_2+(e_0+e_1)e_3)e_5
.\]
Taking a limit that sends $e_5 \to 0$ one sees that the closure of $\Xi$ contains $ \sigma_2^1\Gr(3,6)$. Experiments suggest that that $\Xi$ is not contained in $\tau$. 
Indeed, the Grassmann discriminant \cite{HolweckOedingE8}*{Ex.~6.1}, the defining polynomial for the hypersurface $\tau(\Gr(3,6))$, evaluates at $\xi$ to $15 \not \equiv 0 \mod 2$, hence implying non-membership: $\tau$, i.e. $\Xi \not \subset \tau(\Gr(3,6))$. 
\end{example}

We note a bijection between $\sigma_2^1 (\Gr(3,6))^\circ$ and $\Gr(1,6)^\circ \times \sigma_2(\Gr(2,5))^\circ= (\FF_2^6 \setminus {0}) \times  \PP \bw{2} \FF_2^5 \setminus \Gr(2,5)) $, the fiber bundle from Theorem~\ref{thm:fiber}. The number of points of the latter is, using \eqref{eq:numGr}, $(2^6-1) \cdot (2^{\binom{5}{2}} - \frac{(2^5-1)(2^4-1)}{2^2-1})
= 54,684$, which agrees with the exhaustive count. Further, we have a an identifiability over $\FF_2$ for $\sigma_2^1( \Gr(3,6))^\circ$, whose points correspond uniquely to pairs of a non-zero vector in $\FF^6$ and a full rank skew-symmetric $5\times 5$ matrix over $\FF_2$.

\section*{Acknowledgements}
Oeding thanks Roland Abauf, Elisa Postinghel, for initial discussions on this topic. Bidleman thanks Matt Speck and Colby Muir for discussions on the subject.

\newcommand{\arxiv}[1]{\href{http://arxiv.org/abs/#1}{{\tt arXiv:#1}}}
\bibliographystyle{amsplain}
\bibliography{BidlOed_main_bibfile}

\end{document}